\documentclass[12pt,reqno]{amsart}
\usepackage{amsmath,amssymb,amscd,graphicx}
\usepackage[all]{xy}
\usepackage[latin1]{inputenc}
\usepackage[T1]{fontenc}
\usepackage{hyperref}
\usepackage{ifthen}
\usepackage{fullpage}
\usepackage{subsupscripts}
\usepackage[margin=1in]{geometry}	
\usepackage{url}
\usepackage{cite}
\usepackage{tikz}

\newboolean{workmode}
\setboolean{workmode}{false}

\newboolean{sections}
\setboolean{sections}{true}

\input xy
\xyoption{all}
\xyoption{2cell}
\UseComputerModernTips
\UseTwocells

\makeatletter
\def\th@plain{%
  \thm@notefont{}
  \itshape 
}
\def\th@definition{%
  \thm@notefont{}
  \normalfont 
}
\makeatother

\newcommand{\calA}{\mathcal{A}}
\newcommand{\calB}{\mathcal{B}}

\newcommand{\calE}{\mathcal{E}}
\newcommand{\calF}{\mathcal{F}}
\newcommand{\calG}{\mathcal{G}}
\newcommand{\calH}{\mathcal{H}}

\newcommand{\CC}{{\mathbb{C}}}  
\newcommand{\KK}{{\mathbb{K}}}
\newcommand{\RR}{{\mathbb{R}}}  
\renewcommand{\SS}{{\mathbb{S}}} 
\newcommand{\ZZ}{{\mathbb{Z}}}  

\newcommand{\curv}{{\operatorname{curv}}}   
\newcommand{\Hom}{{\operatorname{Hom}}} 
\newcommand{\id}{{\operatorname{id}}}  
\newcommand{\im}{{\operatorname{im}}}  
\newcommand{\pr}{{\operatorname{pr}}} 

\DeclareMathOperator{\PU}{PU}  
\newcommand{\U}{{\operatorname{U}}}  

\newcommand{\Mfld}{\mathbf{Mfld}} 
\newcommand{\BiCat}{\mathbf{BiCat}} 

\newcommand{\del}{{\partial}} 
\newcommand{\toto}{{~\rightrightarrows~}} 
\DeclareMathOperator{\hhom}{2-Hom}  
\newcommand{\supth}{{^{\text{th}}}} 

\newcommand{\rotateRPY}[3]
{   \pgfmathsetmacro{\rollangle}{#1}
    \pgfmathsetmacro{\pitchangle}{#2}
    \pgfmathsetmacro{\yawangle}{#3}

    \pgfmathsetmacro{\newxx}{cos(\yawangle)*cos(\pitchangle)}
    \pgfmathsetmacro{\newxy}{sin(\yawangle)*cos(\pitchangle)}
    \pgfmathsetmacro{\newxz}{-sin(\pitchangle)}
    \path (\newxx,\newxy,\newxz);
    \pgfgetlastxy{\nxx}{\nxy};

    \pgfmathsetmacro{\newyx}{cos(\yawangle)*sin(\pitchangle)*sin(\rollangle)-sin(\yawangle)*cos(\rollangle)}
    \pgfmathsetmacro{\newyy}{sin(\yawangle)*sin(\pitchangle)*sin(\rollangle)+ cos(\yawangle)*cos(\rollangle)}
    \pgfmathsetmacro{\newyz}{cos(\pitchangle)*sin(\rollangle)}
    \path (\newyx,\newyy,\newyz);
    \pgfgetlastxy{\nyx}{\nyy};

    \pgfmathsetmacro{\newzx}{cos(\yawangle)*sin(\pitchangle)*cos(\rollangle)+ sin(\yawangle)*sin(\rollangle)}
    \pgfmathsetmacro{\newzy}{sin(\yawangle)*sin(\pitchangle)*cos(\rollangle)-cos(\yawangle)*sin(\rollangle)}
    \pgfmathsetmacro{\newzz}{cos(\pitchangle)*cos(\rollangle)}
    \path (\newzx,\newzy,\newzz);
    \pgfgetlastxy{\nzx}{\nzy};
}

\tikzset{RPY/.style={x={(\nxx,\nxy)},y={(\nyx,\nyy)},z={(\nzx,\nzy)}}}

\newcommand{\ifwork}[1]{\ifthenelse{\boolean{workmode}}{#1}{}}
\newcommand{\comment}[1]{}
\newcommand{\mute}[1]{}
\newcommand{\printname}[1]{}

\ifwork{
\renewcommand{\comment}[1]{{\marginpar{*}\ \scriptsize{#1}\ }}
\renewcommand{\mute}[1]{{\scriptsize \ #1\ }\marginpar{\scriptsize muted}}
\renewcommand{\printname}[1]
    {\smash{\makebox[0pt]{\hspace{-1.0in}\raisebox{8pt}{\tiny #1}}}}
}

\newcommand{\labell}[1] {\label{#1} \printname{#1}}

\makeatletter						
\newtheorem*{rep@theorem}{\rep@title}
\newcommand{\newreptheorem}[2]{
\newenvironment{rep#1}[1]{
\def\rep@title{#2 ##1}
\begin{rep@theorem}}
{\end{rep@theorem}}}
\makeatother

\newcommand{\ifsection}[2]{\ifthenelse{\boolean{sections}}{#1}{#2}}

\theoremstyle{plain}
\ifsection{
    \newtheorem{theorem}{Theorem}[section]
    
	\numberwithin{equation}{section}
	\numberwithin{figure}{section}
	\usepackage{chngcntr}
	\counterwithin{table}{section}
}
{
    \newtheorem{theorem}{Theorem}
}

\newreptheorem{theorem}{Theorem}		
\newreptheorem{proposition}{Proposition}
\newreptheorem{corollary}{Corollary}
\newtheorem{proposition}[theorem]{Proposition}

\newtheorem{corollary}[theorem]{Corollary}
\newtheorem{lemma}[theorem]{Lemma}

\theoremstyle{definition}
\newtheorem{definition}[theorem]{Definition}

\newtheorem{remark}[theorem]{Remark}

\newcommand{\calBBS}{{\mathcal{B}^2\mathcal{S}^1}} 
\newcommand{\opDC}{{\operatorname{DC}}} 
\newcommand{\DC}{\mathcal{DC}} 
\newcommand{\dashto}{{\;\dashrightarrow\;}} 
\newcommand{\triv}{{\operatorname{triv}}} 
\newcommand{\DiDo}{{\operatorname{DD}}} 
\newcommand{\frakX}{{\mathfrak{X}}}

\author{Derek Krepski}
\address{Dept.\ of Mathematics, Univ.\ of Manitoba, Winnipeg, Manitoba, Canada R3T 2N2}
\email{Derek.Krepski@umanitoba.ca}
\urladdr{\url{http://server.math.umanitoba.ca/~dkrepski/}}

\author{Jordan Watts}
\address{Dept.\ of Mathematics, Univ.\ of Colorado Boulder, Boulder, CO, USA 80309}
\email{jordan.watts@colorado.edu}
\urladdr{\url{http://euclid.colorado.edu/~jowa8403/}}

\title{Differential cocycles and Dixmier-Douady bundles}
\date{\today}

\thanks{DK is partially supported by an NSERC Discovery Grant.}


\begin{document}

\begin{abstract}
This paper exhibits equivalences of $2$-stacks between certain models of $\SS^1$-gerbes and differential $3$-cocycles.  We focus primarily on the model of Dixmier-Douady bundles, and provide an equivalence   between the $2$-stack of Dixmier-Douady bundles and the 2-stack of differential $3$-cocycles of height $1$, where the `height' is related to the presence of connective structure.  
Differential $3$-cocycles of height $2$ (resp.\ height $3$) are shown to be equivalent to $\SS^1$-bundle gerbes with connection (resp.\ with connection and curving).  These equivalences extend to the equivariant setting of $\SS^1$-gerbes over Lie groupoids.
\end{abstract}

\maketitle

\section{Introduction and Preliminaries}\labell{s:intro}

 Originally due to Giraud \cite{giraud1971}, $\SS^1$-gerbes are `geometric models' representing cohomology classes in $H^3(M;\ZZ)$ of a smooth manifold  $M$; these geometric models are analogous to principal $\SS^1$-bundles over $M$, which by Weil's Theorem \cite{weil1952} represent cohomology classes in $H^2(M;\ZZ)$.  
 There are several concrete constructions for $\SS^1$-gerbes in the literature, and this paper focuses primarily on the model of Dixmier-Douady bundles (DD-bundles), which are locally trivial fibre bundles of $C^*$-algebras, with typical fibre $\KK(\calH)$, the compact operators on a separable complex Hilbert space $\calH$. Other concrete constructions for $\SS^1$-gerbes appearing in the literature are $\SS^1$-bundle gerbes (see \cite{murray1996bundle,hitchin2001lectures,chatterjee1998construction}), $\SS^1$-central extensions of Lie groupoids (see \cite{behrend2011differentiable}), and principal Lie 2-group bundles (see \cite{baez2007higher,wockel2011principal,nikolaus2013four}).

Over a fixed manifold $M$,  these constructions naturally result in bicategories.  In the case of  DD-bundles, 1-arrows or \emph{Morita isomorphisms} $\calE\colon\calA_1 \dashto \calA_2$ are Banach space bundles $\calE \to M$ of fibrewise $(\calA_2,\calA_1)$-bimodules, and 2-arrows $\tau\colon \calE_1 \Rightarrow \calE_2$ are Banach space bundle isomorphisms (see Section \ref{s:dd} for details).   Analogous to the (first) Chern class for principal $\SS^1$-bundles, we can associate to a DD-bundle $\calA\to M$ its Dixmier-Douady class (DD-class)
 $\DiDo(\calA) \in H^3(M;\ZZ)$, and by a theorem of Dixmier and Douady \cite{dixmier1963champs},  Morita isomorphism classes of DD-bundles are classified by their DD-class. The need to relax the notion of isomorphism of DD-bundles, from usual `structure-preserving' fibre bundle isomorphisms  to Morita isomorphisms, is related to the fact that two non-isomorphic (in the sense of fibre bundles) DD-bundles can have the same Dixmier-Douady class.  An alternate fix is to restrict to the case $\dim\calH=\infty$ (see \cite[Theorem 4.85]{raeburn1998morita}); however, there are naturally occurring examples of interest  with finite dimensional fibres (e.g.\ the Clifford algebra bundle of an even rank Euclidean vector bundle).  (For bundle gerbes, the situation is similar ---there are non-isomorphic bundle gerbes with the same DD-class, which is what motivated the definition of \emph{stable isomorphism} (see \cite{murray2000bundle,waldorf2007more}). In that setting, one also has 2-arrows, namely \emph{transformations} of stable isomorphisms to obtain a bicategory of bundle gerbes over a space.)  Additionally, since DD-bundles can be pulled back along smooth maps, we obtain a presheaf $\calBBS$ of bicategories over the category of smooth manifolds $\Mfld$, $M \mapsto \calBBS(M)$ (see Proposition \ref{p:preshfbicat}).

The local nature of the bicategory $\calBBS(M)$ of DD-bundles over $M$ is competently described in the language of stacks.
Roughly speaking, since DD-bundles are locally trivial fibre bundles, the bicategory $\calBBS(M)$ can  be reconstructed from the bicategories $\calBBS(U_\alpha)$, where $\{U_\alpha\}$ is any open cover of $M$. Such a reconstruction, however, should accommodate the more general notion of Morita isomorphism, which should ultimately be used to `glue' together DD-bundles $\calA_\alpha \to U_\alpha$ with (possibly non-isomorphic) fibres $\KK(\calH_\alpha)$.  This notion of gluing is made more precise in Theorem \ref{t:dd}, which states that  the presheaf $\calBBS$ of bicategories  is a 2-stack.  In more detail, following \cite{nikolaus2011equivariance},   introduce   the \emph{descent bicategory} $\calBBS(U_\bullet)$ associated to the cover $\{U_\alpha\}$ of $M$ (see Definition \ref{d:desc categ}), which naturally comes with a functor $\calBBS(M) \to \calBBS(U_\bullet)$ induced by restriction to the open sets in the cover.  That $\calBBS$ is a 2-stack means that this restriction functor is an equivalence of bicategories for every $M$ and every cover of $M$.  

This paper relates  DD-bundles to  \emph{differential $3$-cocycles of height $1$}, following ideas in \cite{lerman2008differential} that considered the case of principal $\SS^1$-bundles and differential $2$-cocycles. In \cite[Section 3.2]{hopkins2005quadratic}, Hopkins and Singer introduce a category of differential $k$-cocycles $\opDC_s^k$ (see Section \ref{ss:cochain}), where $s>0$ is an integer, which we shall refer to as the \emph{height}.  The cochain complexes $\opDC_s^*$   provide a kind of refinement of singular cohomology.  For example, the cohomology group $H^2(\opDC_s^*(M))$ classifies principal $\SS^1$-bundles over $M$  when $s=1$, and when $s=2$ it classifies principal $\SS^1$-bundles with connection (up to connection preserving isomorphism).  An important feature of this complex is that $H^k(\opDC_k^*(M))$ is isomorphic to the group of differential characters, due to Cheeger and Simons \cite{cheeger1985differential}. The perspective from\cite{lerman2008differential} adopted here views the cocycles in $\opDC_s^3(M)$ as objects of a 2-category $\DC_s^3(M)$, where by construction cohomology classes correspond precisely to isomorphism classes of objects (see Section \ref{ss:cochain}). In fact,  since cochains can be pulled back along smooth maps, we have a presheaf of 2-categories $\DC_s^3$, and Theorem \ref{t:dc} verifies that $\DC_s^3$ is a 2-stack, which we refer to as the 2-stack of differential $3$-cocycles of height $s$.  When $s=3$, we will follow \cite{lerman2008differential} and call $\DC_3^3$ the 2-stack of differential characters of degree 3.  Our first main result, Theorem \ref{t:dd-dc}, shows that there is an equivalence of 2-stacks $\calBBS \cong \DC_1^3$.

Theorem \ref{t:dd-dc} has some immediate consequences. First, for any manifold $M$ the equivalence $\calBBS(M) \cong \DC_1^3(M)$  results in a `strictification' of the bicategory of DD-bundles to the strict 2-category of differential cocycles, which can be useful in practice.  (For example, they were used in \cite{krepski2016groupoid} to verify the compatibility among certain definitions of prequantization in the context of Hamiltonian actions of quasi-symplectic/twisted presymplectic groupoids.) 

Second, the equivalence as 2-stacks provides an equivalence of equivariant objects as well.  That is, for any Lie groupoid $\Gamma_1 \toto \Gamma_0$, we may consider the bicategories of $\Gamma_\bullet$-equivariant DD-bundles $\calBBS(\Gamma_\bullet)$ (for an action groupoid $G\times M \toto M$, this is a weakening of the usual notion of $G$-equivariant DD-bundles) and $\Gamma_\bullet$-equivariant differential cocycles $\DC_1^3(\Gamma_\bullet)$ (see Section \ref{ss:equiv} and Definition \ref{d:eqDD}).  The equivalence of 2-stacks automatically gives the corresponding equivalence of equivariant objects $\calBBS(\Gamma_\bullet)\cong\DC_1^3(\Gamma_\bullet)$.  

Third, since the isomorphism classes of objects in the 2-categories $\DC_1^3(M)$ and $\DC_1^3(\Gamma_\bullet)$ are easily computed (by taking cohomology of the corresponding cochain complexes), we obtain the Dixmier-Douady classification of Morita isomorphism classes of DD-bundles over $M$ by $H^3(M;\ZZ)$ (see Corollary \ref{c:iso class}) and its equivariant counterpart, classifying Morita isomorphism classes of $\Gamma_\bullet$-equivariant DD-bundles by $H^3(\Gamma_\bullet;\ZZ)$ for proper Lie groupoids $\Gamma_1 \toto\Gamma_0$ (see Corollary \ref{c:iso class gpd}).  For compact Lie groups $G$, it is well known that $G$-equivariant DD-bundles are classified by $H^3_G(M;\ZZ)$.  An interesting consequence of Corollary \ref{c:iso class gpd} is that every $G\times M \toto M$ equivariant DD-bundle (a weaker notion that the usual notion of $G$-equivariance) is Morita equivalent to a genuine $G$-equivariant DD-bundle over $M$.

We also obtain refinements of Theorem \ref{t:dd-dc} in the setting of $\SS^1$-bundle gerbes with connective structures.
By making use of the technology in \cite{nikolaus2011equivariance} showing that $\SS^1$-bundle gerbes with connection and curving form a 2-stack $\mathrm{Grb}^{\nabla,B}$, our second main theorem (Theorem \ref{t:dcbundlegerbwconnex}) establishes a corresponding equivalence  $\mathrm{Grb}^{\nabla,B} \cong \DC_3^3$ with the 2-stack of differential characters of degree 3.  (Theorem \ref{t:dcbundlegrb} verifies the expected equivalences to differential $3$-cocycles for bundle gerbes without connections, $\mathrm{Grb} \cong \DC_1^3$, and with connections but no specified curving $\mathrm{Grb}^{\nabla} \cong \DC_2^3$.) 

Similar to the corollaries listed above, we immediately obtain the classification of stable isomorphism classes of $\SS^1$-bundle gerbes over $M$ with (or without) connection by $H^3(M;\ZZ)$ (see Corollary \ref{c:iso class grbcon}), and the classification of stable isomorphism classes of $\SS^1$-bundle gerbes with connection and curving by differential characters of degree 3, $H^3(\opDC_3^*(M))$ (see Corollary \ref{c:iso class diffchar}).  The equivariant versions state that  (Corollary \ref{c:iso class grb equiv}) stable isomorphism classes of $\Gamma_\bullet$-equivariant $\SS^1$-bundle gerbes are classified by $H^3(\Gamma_\bullet;\ZZ)$ (for proper Lie groupoids $\Gamma_1\toto \Gamma_0$);   (Corollary \ref{c:iso class grbcon equiv}) stable isomorphism classes of $\Gamma_\bullet$-equivariant bundle gerbes with connection   are classified by $H^3(\opDC_2^*(\Gamma_\bullet))$;  and (Corollary \ref{c:iso class grbconcurv equiv}) $\Gamma_\bullet$-equivariant bundle gerbes with connection and curving  are classified by differential characters $H^3(\opDC_3^*(\Gamma_\bullet))$.

\medskip

The paper is organized as follows. In the remainder of this section, we collect some preliminaries on the simplicial manifold $\Gamma_\bullet$ associated to a Lie groupoid $\Gamma_1 \toto \Gamma_0$, and recall some terminology related to presheaves of bicategories and 2-stacks.  

Section \ref{s:dc} recalls constructions and notation regarding differential cocycles and verifies in Theorem \ref{t:dc} that differential 3-cocycles form a 2-stack.  

In Section \ref{s:dd}, we review some definitions surrounding the bicategory of DD-bundles over a manifold, as well as their equivariant counterparts on a Lie groupoid.  We show in Theorem \ref{t:dd} that DD-bundles form a 2-stack.

Section \ref{s:dd2fun} contains the  main theorems of the paper.  Namely, this section contains Theorem \ref{t:dd-dc}, exhibiting the equivalence between DD-bundles and differential 3-cocycles of height 1, together with the Corollaries mentioned above.  At the end of this section, we establish the refinements of this result to $\SS^1$-bundle gerbes with connective structures mentioned above (Theorems \ref{t:dcbundlegrb} and \ref{t:dcbundlegerbwconnex}).

\noindent \emph{Acknowledgements.} We thank Eugene Lerman for many useful insights and conversations.

\subsection{Lie Groupoids}

We briefly recall some aspects related to the simplicial manifold $\Gamma_\bullet$ associated to a Lie groupoid $\Gamma_1 \rightrightarrows \Gamma_0$, as well as the resulting double complex arising from a presheaf of chain complexes.  Denote the source and target maps by $s,t\colon \Gamma_1 \to \Gamma_0$, respectively, and write multiplication $\Gamma_1 \times_{\Gamma_0} \Gamma_1 \to \Gamma_1$ as $(g_1,g_2)\mapsto g_1g_2$.

For $k \geq 2$, write 
\[
\Gamma_k = \underbrace{\Gamma_1 \times_{\Gamma_0} \Gamma_1 \times_{\Gamma_0} \cdots \times_{\Gamma_0} \Gamma_1}_{k \text{ factors}}
\]
whose elements are $k$-tuples $(g_1, \ldots, g_k)$ of composable arrows (with $s(g_i)=t(g_{i+1})$).  For $0 \leq i \leq k$, let $\partial_i\colon \Gamma_k \to \Gamma_{k-1}$ be the \emph{face maps} given by
\[
\partial_i(g_1, \ldots, g_k) = \left\{ 
\begin{array}{ll}
(g_2, \ldots, g_k) & \text{if }i=0 \\
(g_1, \ldots, g_ig_{i+1},\ldots, g_k) & \text{if }0<i<k \\
(g_1, \ldots, g_{k-1}) & \text{if }i=k. \\
\end{array}
\right.
\]
For convenience, we set $\partial_0=s$ and $\partial_1=t$ on $\Gamma_1$.  It is easily verified that the face maps satisfy the simplicial identities $\partial_i\partial_j = \partial_{j-1} \partial_i$ for $i<j$.  (We will not require degeneracy maps in this paper.)


Let $(A^*,d)$ denote a presheaf of cochain complexes, and consider the double complex $A^*(\Gamma_\bullet)$, depicted below.
\[
\xymatrix{
\vdots & \vdots & \vdots \\
A^2(\Gamma_0)	\ar[r]^\partial \ar[u]^{d} & A^2(\Gamma_1)\ar[r]^\partial \ar[u]^{-d} & A^2(\Gamma_2) \ar[r]^\partial \ar[u]^{d}	& \cdots	\\
A^1(\Gamma_0)	\ar[r]^\partial \ar[u]^{d} & A^1(\Gamma_1)\ar[r]^\partial \ar[u]^{-d} & A^1(\Gamma_2) \ar[r]^\partial \ar[u]^{d}	& \cdots	\\
A^0(\Gamma_0)	\ar[r]^\partial \ar[u]^{d} & A^0(\Gamma_1)\ar[r]^\partial \ar[u]^{-d} & A^0(\Gamma_2) \ar[r]^\partial \ar[u]^{d}	& \cdots	\\
}
\]
The horizontal differential is the alternating sum of pullbacks of face maps, $\partial = \sum (-1)^i \partial_i^*$.  Denote the total complex by 
\(
(A^*(\Gamma_\bullet))_\mathrm{tot}
\)
with $(A^*(\Gamma_\bullet))_\mathrm{tot}^n = \bigoplus_{p+q=n} A^p(\Gamma_q)$, and    total differential $\delta = (-1)^q d \oplus \partial$.

In this paper, we will use the de~Rham complex $\Omega^*$, smooth singular cochains $C^*(-;\ZZ)$ and $C^*(-;\RR)$, and a cochain complex of Hopkins and Singer\cite{hopkins2005quadratic}, denoted $\opDC_s^*$ following the notation in \cite{lerman2008differential} (reviewed in Section \ref{ss:dc}).  Note that we will abuse notation and use integration of forms to view $\Omega^*(M)\subset C^*(M;\RR)$ and also view $C^*(M;\ZZ) \subset C^*(M;\RR)$.

For an open cover $\{U_\alpha\}$ of a manifold $M$, write $U=\coprod_\alpha U_\alpha$ and let $\pi\colon U\to M$ be the natural map induced by inclusions of open sets. We denote the \v{C}ech groupoid $U\times_M U \toto U$ corresponding to the cover $\{U_\alpha\}$ by $U_\bullet$.

\subsection{Presheaves of Bicategories}\labell{ss:bicat}

We recall some details regarding a presheaf in bicategories.  (See \cite{borceaux1994handbook} for background on higher categories.)
Let $\BiCat$ denote the 3-category of bicategories, whose objects are weak 2-categories; 1-arrows are pseudo-functors; 2-arrows are pseudo-natural transformations; and 3-arrows are modifications.  (We shall often omit the prefix \textsl{pseudo} in the text; unless stated otherwise, functors and natural transformations are of the \textsl{pseudo} variety.)

\begin{definition}[Presheaf of bicategories]\labell{d:preshbicat}
A \textbf{presheaf of bicategories} (over manifolds) is a lax functor $\mathfrak{X}\colon \Mfld^{op} \to \BiCat$.  It consists of the following data:
\begin{enumerate}
\item for every manifold $T$, a bicategory $\mathfrak{X}(T)$;
\item for every map $f\colon S\to T$ , a functor $f^*\colon \mathfrak{X}(T) \to \mathfrak{X}(S)$;
\item\labell{item:nat} for every pair of composable maps $R\stackrel{f}{\longrightarrow} S\stackrel{g}{\longrightarrow} T$, a natural isomorphism 
$$
\phi_{f,g}\colon f^*\circ g^* \Rightarrow (gf)^*;
$$
\item\labell{item:mod} for every triple of composable maps $Q\stackrel{f}{\longrightarrow} R\stackrel{g}{\longrightarrow} S\stackrel{h}{\longrightarrow} T$, a modification $\theta=\theta_{f,g,h}$  between the composite natural transformations
$$
f^*g^*h^* \Rightarrow(gf)^*h^* \Rightarrow (hgf)^*
$$
and
$$
f^*g^*h^* \Rightarrow f^*(hg)^* \Rightarrow (hgf)^*.
$$

\end{enumerate}
The modifications $\theta$ are required to satisfy the following coherence condition.
For each quadruple of composable maps $P\stackrel{f}{\longrightarrow}Q\stackrel{g}{\longrightarrow} R\stackrel{h}{\longrightarrow} S\stackrel{k}{\longrightarrow} T$, the two induced modifications between the composite natural transformations,
\begin{equation}\labell{eq:fghk-R}
f^*g^*h^*k^* \Rightarrow f^*g^*(kh)^* \Rightarrow f^*(khg)^* \Rightarrow (khgf)^*,
\end{equation}
and 
\begin{equation}\labell{eq:fghk-L}
f^*g^*h^*k^* \Rightarrow (gf)^*h^*k^* \Rightarrow (hgf)^*k^* \Rightarrow (khgf)^*
\end{equation}
must coincide.
\end{definition}

To elaborate further, the natural transformation in \eqref{item:nat} of Definition \ref{d:preshbicat} above consists of a 1-isomorphism $\phi^A\colon f^*g^*A \to (gf)^*A$ in $\mathfrak{X}(R)$ for each object $A$ in $\mathfrak{X}(T)$ as well as 2-isomorphisms 
$\sigma_F\colon \phi^{A'} \circ f^*g^*F \Rightarrow   (gf)^*F \circ \phi^A$ 
for each 1-arrow $F\colon A\to A'$ in $\mathfrak{X}(T)$. The modifications in \eqref{item:mod} are given as follows.  For each object $A$ in $\mathfrak{X}(T)$ let $\alpha^A=\phi^A_{gf,h}\circ \phi^{h^*A}_{f,g}$  and $\beta^A=\phi^A_{f,hg} \circ f^*\phi^A_{g,h}$ be the 1-isomorphisms in $\mathfrak{X}(Q)$ given by the composite natural transformations in \eqref{item:mod}, and let 
$$
\mu_E\colon \alpha^{A'} \circ f^*g^*h^*E \Rightarrow (hgf)^*E \circ\alpha^A, \quad \nu_E\colon \beta^{A'} \circ f^*g^*h^*E \Rightarrow (hgf)^*E \circ\beta^A
$$
be the natural isomorphisms corresponding to a 1-arrow $E\colon A\to A'$.
The modification $\theta$ consists of 2-arrows $\theta(A)\colon \alpha^A \Rightarrow \beta^A$.  These are required to satisfy the following property.  
For any 2-arrow $\rho\colon E \Rightarrow E'$ between 1-arrows $E, E'\colon A\to A'$,we have the equality
\begin{equation}\labell{eq:mod}
\nu_{E'} \circ (\theta(A') \star f^*g^*h^*\rho) = ((hgf)^*\rho \star \theta(A)) \circ \mu_E, 
\end{equation}
where $\star$ denotes `horizontal' composition of 2-arrows.

The coherence condition on the modifications can be stated as follows.  For any object $A$ in $\mathfrak{X}(T)$, the modifications $\theta$  result in the following two 2-cells from \eqref{eq:fghk-L} to \eqref{eq:fghk-R}: namely, the composition
\begin{equation*} 
\begin{gathered}
\xymatrix@C=1em{
&(gf)^*h^*k^*A \ar[rr] \ar[dr]	\rrtwocell<\omit>{<4>\theta}&	& (hgf)^*k^*A \ar[dr]	 &	\\
f^*g^*h^*k^*A  \ar[ur] \ar[dr]\rrtwocell<\omit>{<0>\sigma}  &	&(gf)^*(kh)^*A   \rtwocell<\omit>{<4>\theta}\ar[rr] & 	& (khgf)^*A \\
	& f^*g^*(kh)^*A \ar[rr] \ar[ur] & &  f^*(khg)^*A \ar[ur] & 
}
\end{gathered}
\end{equation*}
(where $\sigma$ denotes the 2-isomorphism $\sigma_{\phi^A}$) and the composition
\begin{equation*} 
\begin{gathered}
\xymatrix@C=1em{
&(gf)^*h^*k^*A \ar[rr]	\rtwocell<\omit>{<4>\theta}&	& (hgf)^*k^*A \ar[dr]	&	\\
f^*g^*h^*k^*A \ar[rr] \ar[ur] \ar[dr]\rrtwocell<\omit>{<4>{\,\,\,\,\,\,f^*\theta}}  &	&f^*(hg)^*k^*A \ar[dr] \ar[ur] \rrtwocell<\omit>{<0>\theta} & 	& (khgf)^*A \\
	& f^*g^*(hk)^*A \ar[rr] & &  f^*(khg)^*A \ar[ur] & 
}
\end{gathered}
\end{equation*}
(where he have omitted the subscripts on each modification $\theta$).  These 2-cells must agree for every $A$.


\subsection{Equivariant Objects in a Presheaf} \labell{ss:equiv}

Recall the construction from \cite{nikolaus2011equivariance} of equivariant objects in a presheaf of bicategories.

\begin{definition} \labell{d:eqobj}
Let $\mathfrak{X}$ be a presheaf of bicategories over manifolds, and let $\Gamma_1 \rightrightarrows \Gamma_0$ be a Lie groupoid.
The bicategory $\mathfrak{X}(\Gamma_\bullet)$ of \textbf{$\Gamma_\bullet$-equivariant objects of $\mathfrak{X}$} is given by the following:
\begin{enumerate}
\item objects consist of triples $(A, E, \tau)$ where $A$ is an object in $\mathfrak{X}(\Gamma_0)$; $E\colon \partial_0^*A \to \partial_1^*A$ is a 1-isomorphism in $\mathfrak{X}(\Gamma_1)$; and $\tau\colon \partial_2^*E \circ \partial_0^*E \Rightarrow \partial_1^*E$ is a 2-isomorphism in $\mathfrak{X}(\Gamma_2)$ satisfying the coherence condition $\partial_2^*\tau \circ (\id\star \partial_0^*\tau) =\partial_1^*\tau \circ (\partial_3^*\tau \star \id)$ in $\mathfrak{X}(\Gamma_3)$;
\item\labell{ditem:1arrow} 1-arrows $(F,\alpha)\colon (A,E,\tau) \to (A', E', \tau')$ consist of a 1-arrow $F\colon A\to A'$ in $\mathfrak{X}(\Gamma_0)$;
and a 2-arrow $\alpha\colon E' \circ \partial_0^*F \Rightarrow \partial_1^*F \circ E$ in $\mathfrak{X}(\Gamma_1)$ satisfying 
\[
(\id \star \tau) \circ (\partial_2^* \alpha \star \id) \circ (\id \star \partial^*_0 \alpha) = \partial_1^*\alpha \circ (\tau' \star \id)
\]
in $\mathfrak{X}(\Gamma_2)$;
\item\labell{ditem:2arrow} 2-arrows $(F,\alpha) \Rightarrow (F',\alpha')$ consist of a 2-arrow $\beta\colon F\Rightarrow F'$ in $\mathfrak{X}(\Gamma_0)$ satisfying $\alpha' \circ (\id \star \partial_0^*\beta) = (\partial_1^*\beta \star \id) \circ \alpha$ in $\mathfrak{X}(\Gamma_1)$.
\end{enumerate}
\end{definition}

\begin{remark}\labell{r:simpids}
Definition \ref{d:eqobj} implicitly makes use of the simplicial identities on the simplicial manifold $\Gamma_\bullet$ associated to the Lie groupoid $\Gamma_1 \toto \Gamma_0$.  For example, since $\partial_i \circ \partial_j = \partial_{j-1} \circ \partial_i$ ($i<j$),  Definition \ref{d:preshbicat} \eqref{item:nat} gives natural isomorphisms $\chi_{ij}\colon \partial_j^* \circ \partial_i^* \Rightarrow \partial_{i}^* \circ \partial_{j-1}^*$.  A priori, the 2-isomorphism $\tau$ in Definition \ref{d:eqobj} is not well-defined and should instead be written more precisely as a 2-isomorphism $\tau\colon \chi_{12}^A\circ \partial_2^*E \circ (\chi_{02}^A)^{-1} \circ \partial_0^*E \Rightarrow \partial_1^*E \circ (\chi_{01}^A)^{-1}$.  Throughout this paper, we will freely make use of simplicial identities and suppress the resulting 1-isomorhisms $\chi_{ij}$, as in the above definition.
\end{remark}

\subsection{2-Stacks}\labell{ss:stacks}

We briefly recall some notions related to 2-stacks.  For further details, the reader may wish to consult \cite{nikolaus2011equivariance,breen1994classification}.

\begin{definition}\labell{d:desc categ} \cite[Def.~2.12]{nikolaus2011equivariance}
Let $\frakX$ be a presheaf of bicategories over manifolds. 
\begin{enumerate}
	\item  Let $M$ be a manifold.  Given an open cover $\{U_\alpha\}$ of $M$, the \textbf{descent bicategory of $M$ corresponding to the cover $\{U_\alpha \}$} is the bicategory $\frakX(U_\bullet)$ of $U_\bullet$-equivariant objects of $\frakX$.
	\item We say $\frakX$ is a \textbf{2-prestack} (or simply, prestack) if for every manifold $M$ and every open cover $\{U_\alpha\}$ of $M$, the natural restriction functor $\pi^*\colon \frakX(M) \to \frakX(U_\bullet)$ is fully faithful.  
	\item We say $\frakX$ is a \textbf{2-stack} (or simply, stack) if for every manifold $M$ and every open cover $\{U_\alpha\}$ of $M$, the natural restriction functor $\pi^*\colon \frakX(M) \to \frakX(U_\bullet)$ is an equivalence of bicategories.  
\end{enumerate}
\end{definition}

Given a prestack $\frakX_0$, one can associate to it a \textbf{stackification} (see \cite[Section 1.10]{breen1994classification}), which is a stack $\frakX$ together with a morphism (a pseudo-natural transformation) $F\colon \frakX_0 \to \frakX$ such that (\textit{i}) for any $M$, the functor $F_M\colon \frakX_0(M) \to \frakX(M)$ is fully faithful (i.e.~an equivalence on $\Hom$ categories), and (\textit{ii}) every object in $\frakX(M)$ is locally isomorphic to one in the image of $\frakX_0(M)$ (i.e.~for every object $A$ in $\frakX(M)$, there exists a cover $\pi\colon U\to M$ and an object $A_0$ in $\frakX_0(U)$ together with an isomorphism $F_M(A_0) \to \pi^*A$ in $\frakX(U)$). In \cite[Section 3]{nikolaus2011equivariance}, the authors provide a concrete construction for a stackification, called the \emph{plus construction} of $\frakX_0$.

\section{Differential Cocycles }\labell{s:dc}

This section recalls the construction of the (2-)category of \emph{differential cocycles}  from \cite{hopkins2005quadratic}, and establishes some properties analogous to those in \cite{lerman2008differential}, adapted to 2-stacks and differential cocycles of degree 3.  Specifically, differential cocycles of degree 3 are constructed as a presheaf of bicategories (strict 2-groupoids, in fact) associated to a certain presheaf of cochain complexes.  We briefly review the more general construction of
\emph{cocycle 2-categories} associated to any presheaf of cochain complexes in Section \ref{ss:cochain}, turning our attention to the case of differential cocycles in Section \ref{ss:dc}, where we show in Theorem \ref{t:dc} that differential cocycles form a 2-stack.

\subsection{Cocycle $2$-Categories}\labell{ss:cochain}

Following \cite{lerman2008differential}, but adapting to the setting of $2$-categories, we review the construction of a $2$-category from a cochain complex $(A^*,d)$.  In this paper, we will assume all cochain complexes are concentrated in non-negative degrees (i.e.~$A^n=0$ for all $n<0$).

\begin{definition}
\labell{d:cochain}
Let $(A^*,d)$ be a cochain complex of abelian groups.  Fix an integer $k\geq 0$. Define the \textbf{cocycle $2$-category} $\calH^k(A^*)$ as follows:
\begin{enumerate}
\addtocounter{enumi}{-1}
\item objects are $k$-cocycles: $c\in A^k$ such that $dc=0$,
\item a $1$-arrow $c_1\to c_2$ is a $(k-1)$-cochain $b$ such that $c_1-c_2=db$,
\item a $2$-arrow $b_1\Rightarrow b_2$ between two $1$-arrows $b_1,b_2\colon c_1\to c_2$ is an equivalence class $[a]$ of $(k-2)$-cochains such that $b_2-b_1=da$, where $a_1$ is equivalent to $a_2$ if there is a $(k-3)$-cochain $z$ such that $a_2-a_1=dz$.
\end{enumerate}
Composition is given by addition of cochains, and the identity is the $0$-cochain.  It follows that $\calH^k(A^*)$ is a strict $2$-groupoid.
\end{definition}

\begin{remark} \labell{r:basicprops}
Analogous to some of the properties of cochain categories listed in \cite[Section 3.1]{lerman2008differential}, we note that isomorphism classes of objects in $\calH^k(A^*)$ are in one-to-one correspondence with $H^k(A^*)$, and that the automorphism category of any object in $\calH^k(A^*)$ is the cochain category $\calH^{k-1}(A^*)$. 
\end{remark}

\begin{remark}\labell{r:trunc}
Note that the cocycle 2-category $\calH^k(\tau_{k-2}A^*)$ obtained by replacing $A^*$ above with its good truncation 
$$
\tau_{k-2}A^n = \left\{
\begin{array}{rl}
 A^n 			& \text{if }n>k-2 \\
A^{k-2}/\im \, d 	& \text{if }n=k-2 \\
0 			& \text{if }n<k-2.
\end{array}
\right.
$$
is identical to $\calH^k(A^*)$.
\end{remark}

%

This construction behaves well with respect to morphisms of cochain complexes.  In particular, a morphism of cochain complexes $f\colon(A^*,d_A)\to(B^*,d_B)$ naturally induces a 2-functor $\calH^k(f)\colon\calH^k(A^*)\to\calH^k(B^*)$, and a cochain homotopy $s\colon A^* \to B^{*-1}$ between cochain maps $f$ and $g$ induces a pseudo-natural transformation  $\calH^k(s)\colon\calH^k(f)\Rightarrow\calH^k(g)$.

\begin{lemma}\labell{l:presheaf}
Let $(A^*(-),d)$ be a presheaf of complexes of abelian groups over the category $\Mfld$.  Then the assignment $M\mapsto\calH^k(A^*(M))$ is a presheaf of strict 2-groupoids.
\end{lemma}

\begin{proof}
Fix a presheaf of cochain complexes $(A^*(-),d)$ over $\Mfld$.  Then we already have that $\calH^k(A^*(-))$ is a strict 2-groupoid.  For a smooth map $f\colon M\to N$  of manifolds,   we obtain the pullback map $f^*\colon\calH^k(A^*(N))\to\calH^k(A^*(M))$. And for a pair of composable maps $M\stackrel{f}{\longrightarrow} N\stackrel{g}{\longrightarrow} P$,   we have $f^*g^*=(g\circ f)^*\colon A^*(P)\to A^*(M)$ on cochains; therefore, we have trivial natural transformations and, in turn,  trivial modifications.  
\end{proof}

Let $k\geq 0$. Given a presheaf of cochain complexes $(A^*,d)$ and a Lie groupoid $\Gamma_1 \toto \Gamma_0$, consider the bicategory $\calH^k(A^*(\Gamma_\bullet))$ of $\Gamma_\bullet$-equivariant objects of $\calH^k(A^*(-))$.  It is straightforward to verify that there is an isomorphism of bicategories $\calH^k(A^*(\Gamma_\bullet)) \cong \calH^k((\tau_{k-2}{A}^*(\Gamma_\bullet))_\mathrm{tot})$, where $\tau_{k-2}A^*$ is the good truncation of $A^*$ at $k-2$ (cf. Remark \ref{r:trunc}).
In subsequent sections of this paper, we will be particularly interested in the case $k=3$.  If the $0\supth$ cohomology of the presheaf of complexes vanishes identically, a straightforward argument shows the equivalence of bicategories $\calH^3(A^*(\Gamma_\bullet)) \cong \calH^3((\tau_{1}{A}^*(\Gamma_\bullet))_\mathrm{tot})$ can be improved:

\begin{proposition} \labell{p:equiv doublecpx}
Let $(A^*(-),d)$ be a presheaf of cochain complexes on $\Mfld$.  Suppose that  $H^0(A^*(M))=0$ for all manifolds $M$.  Then for any Lie groupoid $\Gamma_1 \toto \Gamma_0$, the bicategory $\calH^3(A^*(\Gamma_\bullet))$ of $\Gamma_\bullet$-equivariant objects of $\calH^3(A^*(-))$ is equivalent to the cocycle 2-category $\calH^3((A^*(\Gamma_\bullet))_\mathrm{tot})$.
\end{proposition}

\subsection{Differentical Cocycles as a $2$-Stack}\labell{ss:dc}

In this section, we review the construction of the  complex of differential cochains on a manifold and the resulting 2-category of differential cocycles (see \cite{hopkins2005quadratic,lerman2008differential}), paying special attention to differential cocycles of degree 3.  In Theorem \ref{t:dc}, we show that the corresponding  presheaf of cocycle 2-categories forms a 2-stack, extending the treatment  in \cite{lerman2008differential} of the degree 2 case.

\begin{definition}
\labell{d:dc}
  Fix an integer $s>0$. Let  $M$ be a manifold.  The complex of \textbf{differential cochains} of $M$ of \textbf{height} $s$, denoted $\opDC^*_s(M)$, is defined as
$$
\opDC^k_s(M):=\{(c,h,\omega)\in C^k(M;\ZZ)\times C^{k-1}(M;\RR)\times\Omega^k(M)\mid \omega=0 \text{ if } k<s\},
$$  
with differential $d\colon\opDC^k_s(M)\to\opDC^{k+1}_s(M)$ given  by 
$$
d(c,h,\omega):=(dc,\omega-c-dh,d\omega).
$$
\end{definition}

Then $\opDC_s^*$ defines a presheaf of cochain complexes on $\Mfld$, and we let $\DC_s^k$ denote the presheaf of bicategories, $\DC^k_s(M):=\calH^k(\opDC^\bullet_s(M))$, of the resulting cocycle 2-category.  Using  the Theorem below, we call $\DC_s^3$ the \emph{$2$-stack of differential cocycles of degree $3$ and height $s$}.

\begin{theorem}
\labell{t:dc}
Let $s>0$.
The presheaf of differential 3-cocycles $\DC^3_s$ is a $2$-stack over $\Mfld$.
\end{theorem}
\begin{proof}
The proof is very similar to the proof of \cite[Prop.~3.4]{lerman2008differential}, which shows that the presheaf $\DC_s^2$ is a 1-stack.  We summarize the main points here.  Let $M$ be a manifold and $\pi\colon U\to M$ a covering.  We need to show that the restriction functor $\pi^*\colon \DC_s^3(M) \to \DC_s^3(U_\bullet)$ is an equivalence.  By Remark \ref{r:trunc} and Proposition \ref{p:equiv doublecpx}, it suffices to verify that the restriction $\calH^3(\tau_1\opDC_s^*(M)) \to \calH^3(\tau_1 \opDC_s^*(U_\bullet)_\mathrm{tot})$ is an equivalence.  Such an equivalence follows directly from the triviality of the cohomology of the double complex
\begin{equation}\labell{e:equiv3}
\xymatrix{
\vdots & \vdots & \vdots &  \\
\opDC_s^3(M) \ar[u]_d \ar[r]^{\pi^*} &\opDC^3_s(U_0) \ar[u]_{d} \ar[r]^{\partial} & \opDC^3_s(U_1) \ar[u]_{-d} \ar[r]^{\partial} &  \dots \\
\opDC_s^2(M) \ar[u]_d \ar[r]^{\pi^*} &\opDC^2_s(U_0) \ar[u]_{d} \ar[r]^{\partial} & \opDC^2_s(U_1) \ar[u]_{-d} \ar[r]^{\partial} &  \dots \\
\widetilde{\opDC}_s^1(M) \ar[u]_d \ar[r]^{\pi^*} &\widetilde{\opDC}^1_s(U_0) \ar[u]_{d} \ar[r]^{\partial} & \widetilde{\opDC}^1_s(U_1) \ar[u]_{-d} \ar[r]^{\partial} &   \dots \\
}
\end{equation}
where $\widetilde{\opDC}_s^1(-) = \opDC_s^1(-)/\im\, d$.
By the acyclic assembly lemma of homological algebra, it suffices to verify that the above rows are exact, which is shown directly in \cite{lerman2008differential}.
\end{proof}

We will be mainly interested in the case of differential 3-cocycles with heights $s=1,2,3$.  We record the following Proposition for later use, whose proof is completely analogous to the one appearing in \cite[Sections 4.2 and 4.3]{lerman2008differential} for differential cocycles of degree 2.

\begin{proposition} \labell{p:degree3coh} \mbox{}
\begin{enumerate}
\item Let $s\in \{1,2\}$. For any manifold $M$, the natural projection  $\pr\colon DC^3_s(M) \to C^3(M;\ZZ)$ induces an isomorphism on cohomology: $H^3(DC_s^*(M)) \cong H^3(M;\ZZ)$.
\item For any proper Lie groupoid $\Gamma_1 \toto \Gamma_0$, the natural projection
\[
 \pr\colon   \bigoplus_{p+q=3} DC_1^p(\Gamma_q) \to \bigoplus_{p+q=3} C^p(\Gamma_q;\ZZ)
 \]
  induces an isomorphism on cohomology: $H^3(DC_1^*(\Gamma_\bullet)_\mathrm{tot}) \cong H^3(\Gamma_\bullet;\ZZ)$.
\end{enumerate}
\end{proposition}

\section{Dixmier-Douady Bundles}\labell{s:dd}

We begin by recalling some definitions surrounding Dixmier-Douady bundles.  For further background, we refer to  \cite{alekseev2012dirac} and \cite{raeburn1998morita}.

\begin{definition}[Dixmier-Douady Bundles]\labell{d:dd}
Fix a manifold $M$.  A \textbf{Dixmier-Douady bundle} (\textbf{DD-bundle}) $\calA\to M$ is a locally trivial bundle of $C^*$-algebras with typical fibre $\KK(\calH)$, the $C^*$-algebra of compact operators on a separable complex Hilbert space $\calH$, and with structure group $\operatorname{Aut}(\KK(\calH))=\PU(\calH)$.  Here, we use the strong operator topology.

A \textbf{Morita isomorphism} of DD-bundles $\calE\colon(\calA_1\to M)\dashto(\calA_2\to M)$ is a locally trivial Banach space bundle $\calE\to M$ with typical fibre  $\KK(\calH_1,\calH_2)$, the compact operators from $\calH_1$ to $\calH_2$ (where the typical fibre of $\calA_i$ is $\KK(\calH_i)$, $i=1,2$).  The bundle $\calE$ comes equipped with a natural fibrewise ($\calA_2,\calA_1$)-bimodule structure 
$$
\calA_2\circlearrowleft\calE\circlearrowright\calA_1,
$$
locally modelled on the natural $(\KK(\calH_2),\KK(\calH_1))$-bimodule structure on $\KK(\calH_1,\calH_2)$ given by post- and pre-composition of operators.  
The composition of two Morita isomorphisms $\calE_1\colon\calA_1\dashto\calA_2$ and $\calE_2\colon\calA_2\dashto\calA_3$ is given by $\calE_2\circ\calE_1=\calE_2\otimes_{\calA_2}\calE_1$,  the fibrewise completion of the (algebraic) tensor product over $\calA_2$.

Given two Morita isomorphisms $\calE_1,\calE_2\colon\calA_1\dashto\calA_2$, a \textbf{$2$-isomorphism} $\tau\colon\calE_1\Rightarrow\calE_2$ is a continuous bundle isomorphism $\tau\colon \calE_1\to\calE_2$ that intertwines the norms and the ($\calA_2,\calA_1$)-bimodule structures. 
\end{definition}

\begin{remark}
One can also define a \textbf{Morita morphism} $(\Phi,\calE)\colon(\calA_1\to M_1)\dashto(\calA_2\to M_2)$ of two DD-bundles $\calA_1\to M_1$ and $\calA_2\to M_2$ as a pair: a continuous map $\Phi\colon M_1\to M_2$, and a Morita isomorphism $\calE\colon\calA_1\dashto\Phi^*\calA_2$.  If one were to view a $2$-stack as a bicategory fibred in $2$-groupoids, then this definition would be required.  However, since we are taking the (equivalent) sheaf perspective of a $2$-stack, we will not need to work with these.  
\end{remark}

\begin{lemma}\labell{l:line bdle}
Let $M$ be a smooth manifold, and let $\calA_1,\calA_2 \to M$ be DD-bundles over $M$.  Suppose that $\calE, \calE'\colon  \calA_1 \dashto \calA_2$ are Morita isomorphisms, and set $L=\Hom_{\calA_2-\calA_1}(\calE,\calE')$ denote the bundle of bimodule homomorphisms $\calE\to \calE'$.  Then 
\begin{enumerate}
\item $L$ is a Hermitian line bundle over $M$; and
\item the fibrewise `evaluation map' $\calE\otimes L \to \calE'$ is an isomorphism of ($\calA_2,\calA_1$)-bimodule bundles.
\end{enumerate}
\end{lemma}
\begin{proof}
That $L$ is a 1-dimensional complex line bundle follows from the fact that any bimodule homomorphism $\KK(\calH_1,\calH_2) \to \KK(\calH_1,\calH_2)$ is a scalar.  
To see this, choose (possibly finite) bases $\{e_1, e_2, \ldots \}$ for $\calH_1$ and $\{f_1, f_2, \ldots \}$ for $\calH_2$.
 Let $\phi\colon \KK(\calH_1,\calH_2)\to \KK(\calH_1,\calH_2)$ be a bimodule homomorphism.  Then for any $a \in \KK(\calH_1,\calH_2)$, $x \in \KK(\calH_2)$, and $y\in \KK(\calH_1)$, we have $\phi(xay)=x\phi(a)y$.  
 Recall that compact operators are the norm closure of the subspace generated by rank one operators. For $u\in \calH_2$ and $v\in \calH_1$, let $u\otimes \overline{v}$ be the rank-one operator $(u\otimes \overline{v})(w) = (w,v)u$, where $(-,-)$ is the inner product on $\calH_1$.
  For $k,l=1,2,\ldots$, write $\phi(f_k\otimes \overline{e_l}) = \sum \varphi^{(k,l)}_{ij} f_i\otimes \overline{e_j}$.  Recall that $(f_i\otimes \overline{f_j})(f_k\otimes \overline{e_l})=f_i\otimes \overline{e_l}$ if $j=k$ and $0$ otherwise, and similarly $(f_k\otimes \overline{e_l})(e_m\otimes \overline{e_n})=f_k\otimes \overline{e_n}$ if $m=n$ and $0$ otherwise.  Then multiplying on the left and right by appropriate elements to isolate coefficients, it is easy to see that the coefficients $\varphi^{(k,l)}_{ij}$ vanish unless $k=i$ and $l=j$.  That is $\phi$ is a diagonal operator.
  A similar argument shows that $\phi$ is a scalar.  
\end{proof}

Any $*$-bundle isomorphism $\phi\colon \calA_1 \to \calA_2$ gives rise to a Morita isomorphism: namely, $\calA_2\colon \calA_1 \dashto \calA_2$ with the natural left $\calA_2$-module structure and right $\calA_1$-module structure induced by $\phi$.  

Recall that given a Morita isomorphism $\calE\colon \calA_1 \dashto \calA_2$, the \textbf{opposite} Morita isomorphism $\calE^*\colon \calA_2 \dashto \calA_1$ is given by $\calE^*=\calE$ as real vector bundles, with opposite (conjugate) scalar multiplication.  There are natural 2-isomorphisms $\calE^* \otimes_{\calA_2} \calE \cong \calA_1$ and $\calE \otimes_{\calA_1} \calE^* \cong \calA_2$.  

It is straightforward to verify that for any manifold $M$, the collection of DD-bundles over $M$ form the objects of a bigroupoid (a weak 2-category in which 1-arrows are coherently invertible and 2-arrows are invertible), with Morita isomorphisms as 1-arrows and 2-isomorphisms as 2-arrows.  Denote this bigroupoid by $\calBBS(M)$.

Given a map $f\colon M_1 \to M_2$ of manifolds, pullbacks of DD-bundles, as well as 1- and 2-arrows are defined in the usual way, resulting in a pseudofunctor  $f^*\colon \calBBS(M_2) \to \calBBS(M_1)$.   In fact, we obtain a (lax) functor $\calBBS$ from $\Mfld^{op}$ to $\mathbf{BiCat}$. 

\begin{proposition} \labell{p:preshfbicat}
The assignment $M\mapsto \calBBS(M)$ defines a presheaf of bicategories.
\end{proposition}
\begin{proof}
Given a pair of composable maps 
$$
R\stackrel{f}{\longrightarrow} S \stackrel{g}{\longrightarrow} T
$$
and a DD-bundle $\calA\to T$, there is a canonical bundle $*$-isomorphism 
$$
\phi_{f,g}\colon f^*(g^*\calA))\to (g f)^*\calA
$$
 and hence a corresponding canonical Morita isomorphism, the  ($(gf)^*\calA,f^*(g^*\calA)$)-bimodule $\calE_{f,g}^\calA=\calE^\calA=(gf)^*\calA$.
Moreover, if $\calF\colon \calA \dashto \calA'$ is a Morita isomorphism, there is a canonical 2-isomorphism
\[
\sigma_\calF\colon  \calE^{\calA'} \otimes_{f^*g^*\calA'} f^*g^*\calF \to (gf)^*\calF \otimes_{(gf)^*\calA} \calE^\calA
\]
induced from the ($\calA',\calA$)-bimodule action on $\calF$.
In other words, we have a natural isomorphism of functors $f^*\circ g^* \cong (g f)^*$.  

Given composable maps 
$$
Q \stackrel{f}{\longrightarrow} R \stackrel{g}{\longrightarrow} S \stackrel{h}{\longrightarrow} T
$$
we show next that there exists a modification between the composite natural isomorphism
\begin{equation}\labell{eq:compleft}
f^*\circ g^*\circ h^* \cong (g f)^* \circ h^* \cong (h g f)^*
\end{equation}
to the composite natural isomorphism
\begin{equation}\labell{eq:compright}
f^* \circ  g^* \circ h^* \cong f^* \circ (h g)^* \cong (h g f)^*.
\end{equation}
Indeed, given a DD-bundle $\calA\to T$, the first composition is given by the Morita isomosphism
$$
\calE_{gf,h}^\calA \otimes_{(gf)^*h^*\calA} \calE_{f,g}^{h^*\calA} = (hgf)^*\calA \otimes_{(gf)^*h^*\calA} (gf)^*h^*\calA 
$$
while the second is given by 
$$
\calE_{f,hg}^\calA \otimes_{f^*(hg)^*\calA} f^*\calE_{g,h}^\calA = (hgf)^*\calA  \otimes_{f^*(hg)^*\calA} f^*(hg)^*\calA.
$$
Each of these is 2-isomorphic (via the respective right-action maps)  to $(hgf)^*\calA$. 
To verify that this results in a family $\theta(\calA)$ of modifications from the
composition \eqref{eq:compleft} to \eqref{eq:compright} (as in Definition \ref{d:preshbicat} \eqref{item:mod}), let $\rho\colon \calF\to \calG$ be a 2-isomorphism between Morita isomorphisms $\calF,\calG\colon \calA \dashto \calA'$.  The required equality \eqref{eq:mod} follows from the commutativity of the diagram below (where we have omitted the subscripts under the $\otimes$ symbols, for simplicity).
\[
\xymatrix@C=5em{
 (hgf)^*\calA' \otimes  (gf)^*h^*\calA' \otimes f^*g^*h^*\calF \ar[d]_{\id \otimes \sigma_{h^*\calF}} \ar[r]^{\theta(\calA')\otimes f^*g^*h^*\rho} & 
 (hgf)^*\calA' \otimes f^*(hg)^*\calA' \otimes f^*g^*h^*\calG \ar[d]^{\id \otimes f^*\sigma_\calG}  \\
 (hgf)^*\calA' \otimes (gf)^*h^*\calF \otimes (gf)^*h^*\calA \ar[d]_{ \sigma_\calF \otimes \id}
  & (hgf)^*\calA' \otimes f^*(hg)^*\calG \otimes f^*(hg)^*\calA \ar[d]^{\sigma_\calG \otimes \id} \\
 (hgf)^*\calF \otimes (hgf)^*\calA \otimes (gf)^*h^*\calA \ar[r]_{(hgf)^*\rho \otimes \theta(\calA)} 
 &
 (hgf)^*\calG \otimes (hgf)^*\calA \otimes f^*(hg)^*\calA 
}
\]

The commutativity follows from the fact that $\rho$ respects the bimodule actions on (pullbacks of) $\calF$ and $\calG$.  The coherence condition is similarly verified; it follows from the axioms of a bimodule action.
\end{proof}

Applying Definition \ref{d:eqobj} to the presheaf $\calBBS$ of DD-bundles over manifolds, we make the following definition.

\begin{definition}[Equivariant DD-bundles] \labell{d:eqDD}Let $\Gamma_1 \rightrightarrows \Gamma_0$ be a Lie groupoid.  
\begin{enumerate}
\item A \textbf{$\Gamma_\bullet$-equivariant DD-bundle} is a triple $(\calA, \calE, \tau)$ consisting of a DD-bundle $\calA\to \Gamma_0$, a Morita isomorphism $\calE\colon \partial_0^*\calA \dashto \partial_1^*\calA$ and a 2-isomorphism $\tau\colon  \partial_2^*\calE \otimes_{\partial_0^*\partial_1^*\calA} \partial_0^*\calE \Rightarrow \partial_1^* \calE$ satisfying the coherence condition $\partial_2^*\tau \circ (\id\star \partial_0^*\tau) =\partial_1^*\tau \circ (\partial_3^*\tau \star \id)$. 

\item A \textbf{$\Gamma_\bullet$-equivariant Morita isomorphism} $(\calE,\calA,\tau) \dashto (\calE',\calA',\tau')$ is a pair $(\calG,\phi)$ consisting of a Morita isomorphism $\calG\colon \calA \dashto \calA'$ and a 2-isomorphism $\phi\colon  \calE' \otimes \partial_0^*\calG \Rightarrow \partial_1^*\calG \otimes \calE$ satisfying a coherence condition $(\id\star\tau) \circ (\partial_2^* \phi \star \id) \circ (\id \star \partial_0^*\phi) = \partial_1^*\phi \circ (\tau' \star \id)$. 
\item A \textbf{$\Gamma_\bullet$-equivariant 2-isomorphism} $(\calG,\phi)\Rightarrow(\calG',\phi')$ is a 2-isomorphism $\beta\colon\calG\Rightarrow\calG'$ satisfying $\phi'\circ(\id\star\del_0^*\beta)=(\del_1^*\beta\star\id)\circ\phi$.

\end{enumerate}
\end{definition}

\begin{remark} Along the lines of Remark \ref{r:simpids}, Definition \ref{d:eqDD} freely uses simplicial identities---for example, by viewing $\partial_1^*\calE$ as a ($\partial_1^*\partial_1^*\calA,\partial_0^*\partial_0^*\calA$)-bimodule and $\partial_2^*\calE$ as a ($\partial_1^*\partial_1^*\calA,\partial_0^*\partial_1^*\calA$)-bimodule.
\end{remark}

\begin{remark} \labell{r:equivariance}
The usual notion of a $G$-equivariant DD-bundle over $M$ is more restrictive than that of a $(G\times M \toto M)$-equivariant DD-bundle.  A DD-bundle $\calA \to M$ equipped with a $G$-action that lifts the $G$-action on $M$ gives an equivariant DD-bundle in the sense of Definition \ref{d:eqDD}, with $\calE=\partial_1^*\calA$ coming from the $*$-isomorphisms given by the $G$-action on $\calA$, and with trivial 2-isomorphism component.
However, as noted in Remark \ref{r:equivariance2}, for compact Lie groups $G$, every $(G\times M \toto M)$-equivariant DD-bundle is Morita isomorphic to a genuine $G$-equivariant DD-bundle.
\end{remark}

We will need the following Lemma for Theorem \ref{t:dd} below.  

\begin{lemma} \labell{l:coaction}
Suppose $g\in \U(\calH)$ implements an automorphism $Ad_g\colon \KK(\calH) \to \KK(\calH)$.   View $\KK(\calH)$ as a $\KK(\calH)$-bimodule, with right action via $Ad_g$. Let $e_1\in \calH$ be a unit vector. The map $\tilde{g}\colon\calH^{op} \to \calH^{op} \otimes_{\KK(\calH)}\KK(\calH)$ given by $\tilde{g}(v) = e_1\otimes (e_1\otimes \bar{v})g^*$ is an isomorphism of $(\CC,\KK(\calH))$-bimodules.  In other words, $\tilde{g}$ fills in the 2-cell:

\mute{muted:

\[
\begin{tikzpicture}[line join = round, line cap = round, > = latex]

\coordinate [label=above:$\CC$] (0) at (1,1.7);
\coordinate [label=left:$\KK(\calH)$] (1) at (0,0);
\coordinate [label=right:$\KK(\calH)$] (2) at (2,0);

\coordinate [label=right:$\calH^{op}$] (t1) at (1.5,1);
\coordinate [label=left:$\calH^{op}$] (t2) at (0.6,1);
\coordinate[label=above:$\tilde{g}$] (tau) at (1,0.4);
\coordinate[label=below:$\KK(\calH)$] (Adg) at (1,0);

\fill (0) circle (0.04cm);
\fill (1) circle (0.04cm);
\fill (2) circle (0.04cm);

\draw[<-] (0)--(1);
\draw[<-] (1)--(2);
\draw[<-] (0)--(2);

\end{tikzpicture}
\]

}

$$\xymatrix@C=1em{
& \CC \dtwocell<\omit>{<0>{\,\,\tilde{g}}} &  \\
\KK(\calH) \ar[ur]^{\calH^{op}} & & \KK(\calH) \ar[ll]^{\KK(\calH)} \ar[ul]_{\calH^{op}}  \\
}$$
(Here, $u\otimes \bar{v}$ denotes the rank 1 operator $w\mapsto (w,v)u$, where $(-,-)$ is the inner product on  $\calH$.) 
\end{lemma}
\begin{proof}
Let $e_1$ be a unit vector in $\calH^{op}$, and let $\tilde{g}$ be as in the statement of the lemma.  Since $\tilde{g}$ is $\CC$-linear, it remains to check that $\tilde{g}$ is a  map of right $\KK(\calH)$-modules. Let $x\in \KK(\calH)$, and observe that for $v\in \calH^{op}$
\begin{align*}
\tilde{g}(v\cdot x) &= \tilde{g}(x^*v) \\
			&= e_1\otimes (e_1\otimes \overline{x^*v})g^* \\
			&= e_1 \otimes (e_1\otimes \overline{v})xg^* \\
			&= \tilde{g}(v) \cdot x 
\end{align*}
A direct calculation shows that $\tilde{g}$ is independent of the choice of unit vector. 
\end{proof}

\begin{theorem}\labell{t:dd}
The presheaf of DD-bundles $\calBBS$ is a 2-stack over $\Mfld$.
\end{theorem}
\begin{proof}
To show that $\calBBS$ is a 2-stack, we verify that for any $M$ and any cover $\pi\colon U\to M$ by open sets $\{U_\alpha\}$ in $M$, the restriction $\pi^*\colon\calBBS(M) \to\calBBS(U_\bullet)$ induces an equivalence of bicategories.

We begin by verifying that $\pi^*$ is fully faithful on $\Hom$ categories (i.e.~bijections on the corresponding 2-morphisms $\hhom$).  
Let $\mathcal{A}$ and $\mathcal{B}$ be DD-bundles over $M$.  Denote the restriction by $(-)\big|_U$.  Let $\mathcal{E},\mathcal{F} \in \Hom(\mathcal{A},\mathcal{B})$, and consider the restriction $\hhom(\mathcal{E},\mathcal{F}) \to \hhom(\mathcal{E}\big|_U,\mathcal{F}\big|_U)$.  This is a bijection because a continuous bundle map is uniquely determined by its restrictions to open sets in a cover that agree on overlaps.  
%
%

Next, we show that the restriction functor is essentially surjective on $\Hom$ categories, which shows that the restriction functor (on bicategories) is fully faithful (and hence $\calBBS$ is a 2-prestack).  Recall a 1-morphism in $\mathcal{A}\big|_U \to \mathcal{B}\big|_U$ in $\calBBS(U_\bullet)$ consists of a collection $(\mathcal{E}_\alpha,\phi_{\alpha\beta})$ of bundles $\mathcal{E}_\alpha \to U_\alpha$ of bimodules together with 2-morphisms $\phi_{\alpha\beta} \colon  \mathcal{E}_\alpha \big|_{U_{\alpha\beta}} \to \mathcal{E}_\beta \big|_{U_{\alpha\beta}}$
satisfying 
\[
\phi_{\beta\gamma}\big|_{U_{\alpha\beta\gamma}} \circ \phi_{\alpha\beta}\big|_{U_{\alpha\beta\gamma}} = \phi_{\alpha\gamma}\big|_{U_{\alpha\beta\gamma}}.
\]
Hence the bundles $\mathcal{E}_\alpha$ glue together to give a 1-morphism $\mathcal{E} \in \Hom(\mathcal{A},\mathcal{B})$; therefore, restriction is essentially surjective on $\Hom$ categories, as desired.

Finally, we  show that $\pi^*$ is an equivalence of bicategories, by showing it is essentially surjective on objects.
Let $(\mathcal{A}_\alpha, \mathcal{E}_{\alpha\beta}, \tau_{\alpha\beta\gamma})$ be an object in $\calBBS(U_\bullet)$, where $\mathcal{A}_\alpha \to U_\alpha$ are DD-bundles, equipped with $\left(\mathcal{A}_\alpha \big|_{U_{\alpha\beta}},\mathcal{A}_\beta \big|_{U_{\alpha\beta}}\right)$-bimodules $\mathcal{E_{\alpha\beta}} \to U_{\alpha\beta}$ and 2-isomorphisms 
\[
\tau_{\alpha\beta\gamma}\colon  \mathcal{E}_{\alpha\beta}\big|_{U_{\alpha\beta\gamma}} \otimes \mathcal{E}_{\beta\gamma}\big|_{U_{\alpha\beta\gamma}} \to \mathcal{E}_{\alpha\gamma}\big|_{U_{\alpha\beta\gamma}}
\]
satisfying the coherence condition ``$\partial \tau =0$''---i.e., such  that the following diagram commutes (over $U_{\alpha\beta\gamma\delta}$):
\begin{equation}\labell{eq:cohtau}
\begin{gathered}
\xymatrix@C=5em{
\calE_{\alpha\beta} \otimes \calE_{\beta\gamma}\otimes \calE_{\gamma\delta} \ar[d]_{\tau_{\alpha\beta\gamma}\otimes \id} \ar[r]^{\id \otimes \tau_{\beta\gamma\delta}} & \calE_{\alpha\beta}\otimes \calE_{\beta\delta} \ar[d]^{\tau_{\alpha\beta\delta}} \\
\calE_{\alpha\gamma}\otimes \calE_{\gamma\delta} \ar[r]^{\tau_{\alpha\gamma\delta}} & \calE_{\alpha\delta}
}
\end{gathered}
\end{equation}
We wish to find a DD-bundle $\mathcal{B}\to M$ and a 1-morphism  $(\mathcal{G}_\alpha,\phi_{\alpha\beta})\colon \mathcal{B}\big|_U \to (\mathcal{A}_\alpha, \mathcal{E}_{\alpha\beta}, \tau_{\alpha\beta\gamma})$.  

Note that it suffices to assume that $\{ U_\alpha \}$ is a good cover.  Indeed, suppose $V$ is a refinement of $U$.  Then by \cite[Lemma 4.3]{nikolaus2011equivariance}, $V_\bullet \to U_\bullet$ is a 
weak equivalence of groupoids,
and hence, by \cite[Theorem 2.1.6]{nikolaus2011equivariance}, the restriction functor $r\colon \calBBS(U_\bullet) \to \calBBS({V_\bullet})$ is fully faithful.  Suppose we have  $\mathcal{B}$ and a morphism $\calB\big|_V\cong (\calB\big|_U)\big|_V \to (\mathcal{A}_\alpha, \mathcal{E}_{\alpha\beta}, \tau_{\alpha\beta\gamma})\big|_V$.  Since $r$ is fully faithful, we get the desired morphism $\calB\big|_U \to (\mathcal{A}_\alpha, \mathcal{E}_{\alpha\beta}, \tau_{\alpha\beta\gamma})$.

The DD-bundle $\mathcal{B}$ will result from an $\SS^1$-valued 2-cocycle, defined by gluing trivial $\KK(\calH)$-bundles (with $\dim\calH = \infty$) over $U_\alpha$ with transition maps $g_{\alpha\beta}\colon U_{\alpha\beta} \to \PU(\calH)$.
In this case, the restriction $\calB\big|_U$ in $\calBBS(U_\bullet)$ may be described as follows. Let $\calB_\alpha = \calB\big|_{U_\alpha} = U_\alpha\times \mathbb{K}(\calH)$.  The transition maps $g_{\alpha\beta}$ give bundle isomorphisms $\calB_\beta \to \calB_\alpha$, and with  corresponding Morita isomorphism $\calB_{\alpha\beta} = \calB_\alpha$ with the right $\calB_\beta$-action obtained via $g_{\alpha\beta}$.  The cocycle condition $g_{\alpha\beta}g_{\beta\gamma}=g_{\alpha\gamma}$ guarantees that the action map 
\[
\lambda_{\alpha\beta\gamma}\colon \calB_{\alpha\beta} \otimes_{\calB_\beta} \calB_{\beta\gamma} \to \calB_{\alpha\gamma}
\]
is a map of ($\calB_\alpha,\calB_\gamma$)-bimodules.  Hence,
$\calB\big|_U =(\calB_\alpha, \calB_{\alpha\beta},\lambda_{\alpha\beta\gamma})$.

%
%

To get the 2-cocycle defining $\mathcal{B}$, begin by  choosing Morita trivializations $\calF_\alpha\colon \calA_\alpha \dashrightarrow \CC$ (using the contractibility of the open sets $U_\alpha$).  Then over each $U_{\alpha\beta}$, we get the pair of Morita isomorphisms,
\[
\calE_{\alpha\beta},\, \calF_\alpha^*\otimes \calF_\beta \colon  \calA_\beta \dashrightarrow \calA_\alpha.
\]
Since we are assuming a good cover, the line bundles $L_{\alpha\beta}=\Hom(\calE_{\alpha\beta}, \calF_\alpha^*\otimes \calF_\beta )$ are trivializable.  Therefore, we may choose 2-isomorphisms (i.e. sections of $L_{\alpha\beta}$) $\sigma_{\alpha\beta}\colon \calE_{\alpha\beta}\to \calF_\alpha^*\otimes \calF_\beta $.  That is, we have the following 2-cell:

\mute{muted:
\[
\begin{tikzpicture}[line join = round, line cap = round, > = latex]

\coordinate [label=above:$\CC$] (0) at (1,1.7);
\coordinate [label=left:$\calA_\alpha$] (1) at (0,0);
\coordinate [label=right:$\calA_\beta$] (2) at (2,0);

\coordinate[label=below:$\sigma_{\alpha\beta}$] (tau) at (1,0.8);

\fill (0) circle (0.04cm);
\fill (1) circle (0.04cm);
\fill (2) circle (0.04cm);

\draw[<-] (0)--(1);
\draw[<-] (1)--(2);
\draw[<-] (0)--(2);

\end{tikzpicture}
\]
}

$$\xymatrix@C=2em{
 & \CC \ar[dl]_{\calF_{\alpha}^{*}}  & \\
\calA_{\alpha}  & & \calA_{\beta} \ar[ll]^{\calE_{\alpha\beta}} \ar[ul]_{\calF_{\beta}} \lltwocell<\omit>{<3>\sigma_{\alpha\beta}\;\;\;\;\;} \\
}$$

Over triple intersections, we get the pair of Morita isomophisms,
\[
\xymatrix{
\calE_{\alpha\beta}\otimes \calE_{\beta\gamma} \ar[r]^-{\tau_{\alpha\beta\gamma}} & \calE_{\alpha\gamma}
},\text{ and}
\]
\[
\xymatrix{
{\calE_{\alpha\beta}\otimes \calE_{\beta\gamma} }\ar[r]^-{\sigma_{\alpha\beta}\otimes \sigma_{\beta\gamma}}
& {\calF_\alpha^*\otimes \calF_\beta \otimes \calF_\beta^*\otimes \calF_\gamma} \ar[r] & {\calF_\alpha^*\otimes \calF_\gamma} \ar[r]^-{\sigma_{\alpha\gamma}^{-1}} & \calE_{\alpha\gamma} 
},
\]
which correspond to the two ways of filling in the 2-cell shown below.

\mute{muted:
\[
\begin{tikzpicture}[line join = round, line cap = round, scale=1.7, > = latex]

\coordinate [label=above:$\calA_\beta$] (0) at (1,1.7);
\coordinate [label=left:$\calA_\gamma$] (1) at (0,0);
\coordinate [label=right:$\calA_\alpha$] (2) at (2,0);

\coordinate [label=above:$\sigma_{\beta\gamma}$] (bg) at (0.7,0.5);
\coordinate [label=above:$\sigma_{\alpha\beta}$] (ab) at (1.3,0.5);
\coordinate [label=above:$\sigma_{\alpha\gamma}$] (ga) at (1,0.1);

\coordinate (c) at (1,0.577);
\fill(c) circle (0.02cm);
\draw [->] (0)--(c);
\draw [->] (1)--(c);
\draw [->] (2)--(c);
\fill (0) circle (0.02cm);
\fill (1) circle (0.02cm);
\fill (2) circle (0.02cm);

\draw[<-] (0)--(1);
\draw[->] (1)--(2);
\draw[->] (0)--(2);

\begin{scope}[xshift=1.5in]

\coordinate [label=above:$\calA_\beta$] (0) at (1,1.7);
\coordinate [label=left:$\calA_\gamma$] (1) at (0,0);
\coordinate [label=right:$\calA_\alpha$] (2) at (2,0);

\coordinate[label=below:$\tau_{\alpha\beta\gamma}$] (tau) at (1,0.8);

\fill (0) circle (0.02cm);
\fill (1) circle (0.02cm);
\fill (2) circle (0.02cm);

\draw[<-] (0)--(1);
\draw[->] (1)--(2);
\draw[->] (0)--(2);

\end{scope}

\end{tikzpicture}
\]

}

$$\xymatrix@C=4em@R=4em{
 & \calA_{\beta} \ar[ddr] \ddrtwocell<\omit>{<3>\;\;\;\;\;\sigma_{\alpha\beta}} \ar[d] & & & \calA_{\beta} \ar[ddr] & \\
 & \CC & & & & \\
\calA_{\gamma} \ar[uur] \ar[ur] \ar[rr] \uurtwocell<\omit>{<2>\;\;\;\;\sigma_{\beta\gamma}} & & \calA_{\alpha} \ar[ul] \lltwocell<\omit>{<5>\sigma_{\alpha\gamma}\;\;\;\;\;\;} & \calA_{\gamma} \ar[uur] \ar[rr] \rrtwocell<\omit>{<-10>\;\;\;\;\;\;\;\tau_{\alpha\beta\gamma}} & & \calA_{\alpha} 
}
$$

That is, we have two sections of the line bundle $L'_{\alpha\beta\gamma}=\Hom(\calE_{\alpha\beta}\otimes \calE_{\beta\gamma}, \calE_{\alpha\gamma})$, which must therefore differ by an $\SS^1$-valued function $s_{\alpha\beta\gamma}\colon U_{\alpha\beta\gamma} \to \SS^1$ defined by:
\[
\tau_{\alpha\beta\gamma} = s_{\alpha\beta\gamma} \sigma_{\alpha\gamma}^{-1} \circ (\sigma_{\alpha\beta}\otimes \sigma_{\beta\gamma}).
\]
We claim that $s_{\alpha\beta\gamma}$ defines a 2-cocycle. To see this, consider $\rho_{\alpha\beta\gamma}=\sigma_{\alpha\gamma}^{-1} \circ (\sigma_{\alpha\beta}\otimes \sigma_{\beta\gamma})$.  A direct calculation shows that $\rho$ also satisfies the coherence condition ``$\partial \rho=0$'' (similar to \eqref{eq:cohtau}).  Hence taking $\partial$ of both sides of the equation above gives the desired cocycle condition
\[
s_{\alpha\beta\gamma}s_{\alpha\gamma\delta} = s_{\alpha\beta\delta}s_{\beta\gamma\delta}.
\]
Let $\calB\to M$ be a DD-bundle defined by this 2-cocycle.  

For later use, we note that since $\{ U_\alpha \}$ is a good cover, we can find lifts $\hat{g}_{\alpha\beta}\colon U_{\alpha\beta} \to \U(\calH)$ (with $Ad_{\hat{g}_{\alpha\beta}}=g_{\alpha\beta}$) so that
\begin{align} \labell{eq:lift}
(\hat{g}_{\alpha\beta}\hat{g}_{\beta\gamma})^* = s_{\alpha\beta\gamma}\hat{g}_{\alpha\gamma}^*.
\end{align}
(A priori, \eqref{eq:lift} may only be true up to a coboundary; however, one can choose lifts that give equality on the nose.  See the proof of \cite[Proposition 4.83]{raeburn1998morita}.)


Next we construct an isomorphism $(\calG,\phi)\colon (\calB_\alpha, \calB_{\alpha\beta},\lambda_{\alpha\beta\gamma})\to  (\mathcal{A}_\alpha, \mathcal{E}_{\alpha\beta}, \tau_{\alpha\beta\gamma}).$  Since each $\calB_\alpha$ is a trivial $\mathbb{K}(\calH)$-bundle, we have canonical Morita isomorphisms $\calH^{op}\colon \calB_\alpha \dashrightarrow \CC$ (recall $\calH^{op}$ is the trivial $\calH$-bundle over $U_\alpha$ with  conjugate scalar multiplication, viewed as a ($\mathbb{C},\mathbb{K}(\calH)$)-bimodule.) Let 
\[
\calG_\alpha = \calF_\alpha^* \otimes \calH^{op}\colon  \calB_\alpha \dashrightarrow \calA_\alpha.
\]

To construct $\phi$, we use  Lemma \ref{l:coaction}, to get the  $(\CC,\calB_\beta)$-bimodule map
\[
\tilde{g}_{\alpha\beta}\colon \calH^{op} \to \calH^{op} \otimes_{\calB_\alpha} \calB_{\alpha\beta}.
\]
 Let $\phi_{\alpha\beta}\colon \calE_{\alpha\beta} \otimes \calG_\beta \to \calG_\alpha\otimes \calB_{\alpha\beta}$ denote the 2-morphism given by  the interior of the 2-cell below. 

\mute{muted:
\[
\begin{tikzpicture}[line join = round, line cap = round, scale=3, > = latex]
\coordinate[label=left:$\calB_\beta$] (0) at (0,0);
\coordinate[label=above:$\calA_\beta$] (0p) at (1,1);
\coordinate[label=right:$\calA_\alpha$] (1p) at (2,0);
\coordinate[label=below:$\calB_\alpha$] (1) at (1,-1);
\coordinate (c) at (1,0);

\coordinate[label=above:$\CC$] (lc) at (0.9,0);
\coordinate[label=above:$\calG_{\beta}$] (gb) at (0.45,0.55);
\coordinate[label=above:$\calE_{\alpha\beta}$] (ab) at (1.55,0.55);
\coordinate[label=above:$\calH^{o\!p}$] (h) at (0.5,0);
\coordinate[label=below:$\calF_{\alpha}$] (fa) at (1.5,0);
\coordinate[label=left:$\calF_{\beta}$] (fb) at (1,0.45);
\coordinate[label=right:$\calH^{o\!p}$] (hh) at (1,-0.45);
\coordinate[label=below:$\calG_{\alpha}$] (ga) at (1.55,-0.55);
\coordinate[label=below:$\calB_{\alpha\beta}$] (ba) at (0.45,-0.55);
\coordinate[label=above:$\tilde{g}_{\alpha\beta}$] (gab) at (0.65,-0.45);
\coordinate[label=below:$\sigma_{\alpha\beta}$] (gab) at (1.35,0.45);

\fill (0) circle (0.02cm);
\fill (0p) circle (0.02cm);
\fill (1) circle (0.02cm);
\fill (1p) circle (0.02cm);
\fill (c) circle (0.02cm);

\draw[->] (0)--(1);
\draw[->] (1)--(1p);
\draw[->] (0)--(0p);
\draw[->] (0p)--(1p);

\draw[->] (0)--(c);
\draw[->] (1)--(c);
\draw[->] (0p)--(c);
\draw[->] (1p)--(c);

\end{tikzpicture}
\]
}

$$\xymatrix@C=4em@R=4em{
 & \calA_\beta \ar[d]_{\calF_{\beta}} \ar[dr]^{\calE_{\alpha\beta}} \drtwocell<\omit>{<4>\;\;\;\;\;\sigma_{\alpha\beta}} & \\
\calB_\beta \ar[ur]^{\calG_{\beta}} \ar[r]^{\calH^{op}} \rtwocell<\omit>{<4>\;\;\;\;\;\tilde{g}_{\alpha\beta}} \ar[dr]_{\calB_{\alpha\beta}} & \CC & \calA_\alpha \ar[l]^{\calF_{\alpha}} \\
 & \calB_\alpha \ar[u]_{\calH^{op}} \ar[ur]_{\calG_{\alpha}} & 
}$$

We claim that the coherence condition on $\phi$ (see Definition \ref{d:eqobj} \eqref{ditem:2arrow})  is satisfied, and hence $(\calG,\phi)$ defines an isomorphism. Indeed, the coherence condition amounts to the commutativity of the following diagram.
\begin{equation}
\begin{gathered}
\xymatrix{
\calE_{\alpha\beta} \otimes \calE_{\beta\gamma} \otimes \calF_\gamma^* \otimes \calH^{op} \ar[r]^{\id \otimes \phi_{\beta\gamma}} \ar[d]_{\tau_{\alpha\beta\gamma}\otimes \id \otimes \id}
&
\calE_{\alpha\beta} \otimes \calF^*_\beta \otimes \calH^{op} \otimes \calB_{\beta\gamma} \ar[r]^{\phi_{\alpha\beta} \otimes \id}
&
\calF_\alpha^*\otimes \calH^{op} \otimes \calB_{\alpha\beta} \otimes \calB_{\beta\gamma} \ar[d]^{\lambda_{\alpha\beta\gamma}}
\\
\calE_{\alpha\gamma} \otimes \calF_\gamma^*\otimes \calH^{op} \ar[rr]^{\phi_{\alpha\gamma}}
&
&
\calF_\alpha^* \otimes \calH^{op} \otimes \calB_{\alpha\gamma} 
}
\end{gathered}
\end{equation}
To see that the diagram commutes, observe that the composition along the top can be written as
\[
\xymatrix{
\calE_{\alpha\beta} \otimes \calE_{\beta\gamma} \otimes \calF_\gamma^* \otimes \calH^{op} \ar[r]^-{\sigma_{\alpha\beta}\otimes \sigma_{\beta\gamma} \otimes \id \otimes \tilde{g}_{\beta\gamma}} 
& 
*[r]{\calF_\alpha^*\otimes  \calF_\beta \otimes \calF_\beta^* \otimes \calF_\gamma \otimes \calF_\gamma^* \otimes \calH^{op} \otimes \calB_{\beta\gamma}} \ar@(d,u)[]!<10ex,-2ex>;[dl]!<1ex,2ex>_{\mathsf{pair}} \\
\calF_\alpha^*\otimes \calH^{op} \otimes \calB_{\beta\gamma} \ar[r]^-{\id \otimes \tilde{g}_{\alpha\beta} \otimes \id}
& \calF_\alpha^*\otimes \calH^{op} \otimes \calB_{\alpha\beta} \otimes \calB_{\beta\gamma}
\ar[r]^-{\id \otimes \id \otimes \lambda_{\alpha\beta\gamma}} & 
\calF_\alpha^*\otimes \calH^{op} \otimes \calB_{\alpha\gamma}
}
\]
(where the curved arrow $\mathsf{pair}$ uses the natural pairing given by the fibrewise inner product), while the composition along the bottom can be written as
\[
\xymatrix@C=4em{
\calE_{\alpha\beta} \otimes \calE_{\beta\gamma} \otimes \calF_\gamma^* \otimes \calH^{op} \ar[r]^{\tau_{\alpha\beta\gamma}\otimes \id \otimes \tilde{g}_{\alpha\gamma}}
& \calE_{\alpha\gamma} \otimes \calF_\gamma^*\otimes \calH^{op} \otimes \calB_{\alpha\gamma} \ar@(d,u)[dl]_{\sigma_{\alpha\gamma \otimes \id \otimes \id \otimes \id}} \\
\calF_\alpha^* \otimes \calF_\gamma \otimes \calF_\gamma^* \otimes \calH^{op} \otimes \calB_{\alpha\gamma} \ar[r]^-{\mathsf{pair}}
& \calF_\alpha^*\otimes \calH^{op} \otimes \calB_{\alpha\gamma}
}
\]
Comparing the above compositions, we see that they agree because equation \eqref{eq:lift} holds. 

It follows that the presheaf $\calBBS$ of DD-bundles is a 2-stack.
\end{proof}

\section{The Dixmier-Douady $2$-Functor}\labell{s:dd2fun}

In this section we prove the main result of this paper, Theorem \ref{t:dd-dc}, which states that the 2-stack of Dixmier-Douady bundles is equivalent to the 2-stack of differential 3-cocycles.  We also relate differential 3-cocycles with $\SS^1$-bundle gerbes, and establish  analogous results for $\SS^1$-bundle gerbes with connection and curving.

\begin{definition}
\labell{trivial BS}
Let $\calBBS_\triv$ be the presheaf of bicategories with a  single object in each $\calBBS_\triv(M)$, the trivial DD-bundle $\calA_0:=M\times\KK(\calH)$ (with $\dim \calH=\infty)$, with $1$-arrows all Morita isomorphisms from $\calA_0$ to itself, along with all associated $2$-arrows. That is, the morphism category over $M$ is $\Hom_{\calBBS(M)}(\calA_0,\calA_0)$.
\end{definition}

\begin{lemma}\labell{l:prestack}
$\calBBS_\triv$ is a prestack with stackification $\calBBS$.
\end{lemma}

\begin{proof}
This follows directly from Theorem~\ref{t:dd}.
\end{proof}

\begin{theorem}
\labell{t:dd-dc}
The 2-stack  $\calBBS$ of DD-bundles   is equivalent to the 2-stack $\DC^3_1$ of differential 3-cocycles  of height 1.
\end{theorem}

\begin{proof}
Fix an infinite dimensional separable Hilbert space $\calH$. Let $M$ be a manifold, and let $\calA_0$ be the trivial DD-bundle as in Definition \ref{trivial BS}.
By Lemma \ref{l:line bdle}, we have an equivalence $\Hom_{\calBBS(M)}(\calA_0,\calA_0)\cong \mathcal{BS}^1(M)$, where $\mathcal{BS}^1(M)$ denotes the category of principal $\SS^1$-bundles over $M$.  Define a morphism of 2-prestacks $\DiDo_{\triv}\colon  \calBBS_\triv \to \DC_1^3$ by setting $\DiDo_\triv(M)(\calA_0) = \underline{0}=(0,0,0)$, while on morphisms, $\DiDo_\triv(M)=\mathrm{Ch}(M)\colon \mathcal{BS}^1(M) \to \Hom_{\DC_1^3(M)}(\underline{0},\underline{0})= \DC_1^2(M)$ (see Remark \ref{r:basicprops}), where $\mathrm{Ch}$ is the equivalence of stacks in \cite{lerman2008differential}. Since $\DiDo_\triv$ is fibrewise fully faithful, and every object in $\DC_1^3$ is locally isomorphic to one in the image of $\calBBS_\triv$ (since isomorphism classes of objects in $\DC_1^3(M)$ are classified by $H^3(\opDC_1^3(M))\cong H^3(M;\ZZ)$, by Proposition \ref{p:degree3coh}, and $M$ can be covered by contractible open sets), then $\DC_1^3$ is a 2-stackification of $\calBBS_\triv$ (see \cite[Section 1.10]{breen1994classification}).  Since $\calBBS$ is a 2-stackification of $\calBBS_\triv$, the morphism $\DiDo_\triv$ extends to an equivalence $\DiDo\colon \calBBS \to \DC_1^3$ of 2-stacks.
\end{proof}


Theorem \ref{t:dd-dc} has some immediate consequences.  For any manifold $M$, the DD-functor gives an equivalence of bicategories $\calBBS(M) \cong \DC^3_1(M)$.  Hence,

\begin{corollary} \labell{c:iso class}
Let $M$ be a manifold. There is a one-to-one correspondence between Morita isomorphism classes of DD-bundles over $M$ and $H^3(M;\ZZ)$.
\end{corollary}
\begin{proof}
Isomorphism classes of objects in $\DC^3_1(M)$ (see Remark \ref{r:basicprops}) are in one-to-one correspondence with $H^3(\opDC_1^\bullet(M)) \cong H^3(M;\ZZ)$, by Proposition \ref{p:degree3coh}.
\end{proof}

More generally, for a Lie groupoid $\Gamma_1 \toto \Gamma_0$, the DD-functor gives an equivalence of bicategories $\calBBS(\Gamma_\bullet) \cong \DC^3_1(\Gamma_\bullet)$.  Hence,

\begin{corollary} \labell{c:iso class gpd}
Let $\Gamma_1 \toto \Gamma_0$ be a proper Lie groupoid.  There is a one-to-one correspondence between Morita isomorphism classes of $\Gamma_\bullet$-equivariant DD-bundles and $H^3(\Gamma_\bullet;\ZZ)$.
\end{corollary}
\begin{proof}
Isomorphism classes of objects in $\DC^3_1(\Gamma_\bullet)=\calH^3(\opDC_1^*(\Gamma_\bullet))$ (see Remark \ref{r:basicprops}) are in one-to-one correspondence with $H^3((\opDC_1^*(\Gamma_\bullet))_\mathrm{tot})\cong H^3(\Gamma_\bullet;\ZZ)$, by  Proposition~\ref{p:equiv doublecpx} and Proposition \ref{p:degree3coh}.
\end{proof}

\begin{remark}\labell{r:actual dd class}
Composing the DD-functor with $\pr$ from Proposition~\ref{p:degree3coh}, for any DD-bundle we obtain the classical DD-class in $H^3(M;\ZZ)$ (or $H^3(\Gamma_\bullet;\ZZ)$ in the equivariant case).
\end{remark}

In the case that $\Gamma_1 \toto \Gamma_0$ is an action groupoid corresponding to an action of a Lie group on a manifold, we obtain:

\begin{corollary} \labell{c:iso class act}
Let $G$ be a Lie group acting smoothly and properly on a manifold $M$.  There is a one-to-one correspondence between Morita isomorphism classes of $(G \times M \toto M)$-equivariant DD-bundles and (the simplicial model for) $G$-equivariant cohomology $H^3_G(M;\ZZ)$.
\end{corollary}

\begin{remark} \labell{r:equivariance2}
For compact Lie groups $G$, it is well known that (ordinary) $G$-equivariant DD-bundles over $M$ are classified by $H^3_G(M;\ZZ)$.  Corollary \ref{c:iso class act} implies that every $(G\times M \toto M)$-equivariant bundle is Morita isomorphic to a genuine $G$-equivariant DD-bundle.
\end{remark}

\subsection*{Differential Characters, Bundle Gerbes, and Connections}\labell{s:comparisons}

In this section, we briefly comment on the equivalence of Dixmier-Douady bundles and $\SS^1$-bundle gerbes, and provide refinements of Theorem \ref{t:dd-dc} for bundle gerbes with connective structures.  Bundle gerbes were defined by Murray \cite{murray2000bundle} (see also \cite{chatterjee1998construction,hitchin2001lectures,hitchin2003communications} and \cite{stevenson2000geometry}  for details).  In \cite{nikolaus2011equivariance}, the authors show that bundle gerbes form a 2-stack, namely the 2-stackification of the pre-stack $\mathrm{Grb}_\triv$, which by definition is the presheaf of bicategories with a single object in $\mathrm{Grb}_\triv(M)$ and morphisms the category  $\mathcal{BS}^1(M)$ of principal $\SS^1$-bundles over $M$.  In particular, since the prestacks $\mathrm{Grb}_\triv$ and $\calBBS_\triv$ are equivalent, we immediately obtain the equivalence between $\SS^1$-bundle gerbes $\mathrm{Grb}$ and DD-bundles $\calBBS$.  
This gives part (1) of the Theorem below:

\begin{theorem} \labell{t:dcbundlegrb}
\mbox{}
\begin{enumerate}
\item
The 2-stack  $\mathrm{Grb}$ of $\SS^1$-bundle gerbes is equivalent to the 2-stack $\DC^3_1$ of differential $3$-cocycles of height 1.
\item
The 2-stack  $\mathrm{Grb}^\nabla$ of $\SS^1$-bundle gerbes with connection (but no specified curving) is equivalent to the 2-stack $\DC^3_2$ of differential $3$-cocycles of height 2.
\end{enumerate}
\end{theorem}
\begin{proof}[Proof of (2)]
 The 2-stack $\mathrm{Grb}^\nabla$ is the 2-stackification of the presheaf of bicategories $\mathrm{Grb}^\nabla_\triv$, which by definition is the presheaf of bicategories with a single object in $\mathrm{Grb}^\nabla_\triv(M)$ and morphisms the category  $\mathcal{DBS}^1(M)$ of principal $\SS^1$-bundles over $M$ with connection.  The rest of the argument is the same as in the proof of Theorem \ref{t:dd-dc}, using instead the equivalence $\mathcal{DBS}^1 \cong \DC_2^2$ from \cite{lerman2008differential}.
\end{proof}

Similar to the corollaries following Theorem \ref{t:dd-dc}, we obtain the following:

\begin{corollary} \labell{c:iso class grbcon}
Let $M$ be a manifold. There is a one-to-one correspondence between (stable) isomorphism classes of $\SS^1$-bundle gerbes over $M$ (with or without connection) and $H^3(M;\ZZ)$.
\end{corollary}

\begin{corollary} \labell{c:iso class grb equiv}
Let $\Gamma_1\toto \Gamma_0$ be a proper Lie groupoid.  There is a one-to-one correspondence between (stable) isomorphism classes of $\Gamma_\bullet$-equivariant $\SS^1$-bundle gerbes without connection and $H^3(\Gamma_\bullet;\ZZ)$.  
\end{corollary}

\begin{corollary} \labell{c:iso class grbcon equiv}
Let $\Gamma_1\toto \Gamma_0$ be a Lie groupoid.  There is a one-to-one correspondence between (stable) isomorphism classes of $\Gamma_\bullet$-equivariant $\SS^1$-bundle gerbes with connection and $H^3(\opDC_2^*(\Gamma_\bullet)_{\mathrm{tot}})$.  
\end{corollary}

In fact, the argument used to prove the equivalence of Theorem \ref{t:dd-dc} (and hence also Theorem \ref{t:dcbundlegrb}) can be used to show an equivalence between differential characters of degree 3 and bundle gerbes with connection and curving.  

\begin{theorem} \labell{t:dcbundlegerbwconnex}
The 2-stack of $\SS^1$-bundle gerbes with connection and curving $\mathrm{Grb}^{\nabla,B}$ is equivalent to the 2-stack  $\DC_3^3$ of differential characters of degree 3.
\end{theorem}
\begin{proof}
Recall from \cite{nikolaus2011equivariance}, that the 2-stack of bundle gerbes with connection and curving, denoted $\mathrm{Grb}^{\nabla,B}$ is a stackification of a 2-prestack $\mathrm{Grb}^{\nabla,B}_\triv$, defined as follows.  Objects in $\mathrm{Grb}^{\nabla,B}_\triv(M)$ are 2-forms $\beta \in \Omega^2(M)$; 1-arrows are line bundles over $M$ with connection $(L,\nabla)\colon \beta_1\to \beta_2$ where $\curv(\nabla)=\beta_2-\beta_1$; 2-arrows are connection-preserving bundle maps $(L,\nabla)\to(L',\nabla')$.  

Similar to the proof of Theorem \ref{t:dd-dc}, we define a morphism of 2-prestacks $\mathrm{F}_\triv\colon \mathrm{Grb}_\triv^{\nabla,B} \to \DC_3^3$ as follows. Let $M$ be a manifold.  For an object $\beta \in \Omega^2(M)$ of $\mathrm{Grb}_\triv^{\nabla,B}(M)$, set $\mathrm{F}_\triv(\beta)=(0,\beta,d\beta)$.  On morphism categories, we use the equivalence $\mathrm{DCh}:\mathcal{DBS}^1 \to \DC_2^2$   from \cite[Section 4.4]{lerman2008differential}, between $\SS^1$-bundles with connection and differential characters of degree 2. In particular, send $(L,\nabla)\colon \beta_1\to \beta_2$ to $(c,h,0) \in \opDC_3^2(M)$, where $\mathrm{DCh}(L,\nabla) = (c,h,\beta_2-\beta_1)$, and send $\phi:(L_1, \nabla_1) \to (L_2, \nabla_2)$ to $\mathrm{DCh}(\phi) \in \opDC_2^1(M)=\opDC_3^1(M)$.  

That $F_\triv$ is an equivalence on morphism categories follows from the equivalence $\mathrm{DCh}$.  Indeed, given a morphism $(c,h,0):(0,\beta_1,d\beta_1) \to (0,\beta_2,d\beta_2)$ in $\opDC_3^2(M)$, $(c,h,\beta_2-\beta_1)$ is a cocycle in $\opDC_2^2(M)$ (i.e.\ an object in $\DC_2^2(M)$); therefore, there exists a line bundle with connection  $(L,\nabla)$ so that $\mathrm{DCh}(L,\nabla)=(c',h',\beta_2-\beta_1)$ and a morphism $(b,g,0):(c',h',\beta_2-\beta_1) \to (c,h,\beta_2-\beta_1)$ in $\opDC_2^1(M)$, which may be viewed as a 2-arrow in $\DC_3^1(M)$, $(b,g,0):(c,h,0)\Rightarrow (c',h',0)$, and thus $F_\triv$ is essentially surjective on morphism categories.  That it is fully faithful on morphism categories is clear, since this is true for $\mathrm{DCh}$.

Since $\mathrm{F}_\triv$ is fibrewise fully faithful, and every object in $\DC_3^3$ is locally isomorphic to one in the image of $\mathrm{F}_\triv$ (since isomorphism classes of objects in $\DC_3^3(M)$ are classified by $H^3(\opDC_3^*(M))$, and $M$ can be covered by contractible open sets $U$ on which $H^3(\opDC^*_3(U))=0$), then $\DC_3^3$ is a 2-stackification of $\mathrm{Grb}_\triv^\nabla$ (see \cite[Section 1.10]{breen1994classification}).  Since $\mathrm{Grb}^{\nabla,B}$ is a 2-stackification of $\mathrm{Grb}_\triv^{\nabla,B}$, the morphism $\mathrm{F}_\triv$ extends to an equivalence $\mathrm{F}\colon \mathrm{Grb}^{\nabla,B}\to \DC_3^3$ of 2-stacks.
\end{proof}

Similar to above,   we obtain the following corollaries to Theorem~\ref{t:dcbundlegerbwconnex}.

\begin{corollary}\labell{c:iso class diffchar}
Let $M$ be a manifold. There is a one-to-one correspondence between (stable) isomorphism classes of $\SS^1$-bundles gerbes with connection and curving and $H^3(\opDC_3^*(M))$ differential characters of degree 3.
\end{corollary}

\begin{corollary} \labell{c:iso class grbconcurv equiv}
Let $\Gamma_1\toto \Gamma_0$ be a Lie groupoid.  There is a one-to-one correspondence between (stable) isomorphism classes of $\Gamma_\bullet$-equivariant $\SS^1$-bundle gerbes with connection and curving and $H^3(\opDC_3^*(\Gamma_\bullet)_{\mathrm{tot}})$ equivariant differential characters of degree 3.
\end{corollary}


%

\end{document}